\documentclass[12pt, reqno]{amsart}
\usepackage{amssymb}
\usepackage{amsmath,amssymb,color}
\usepackage{hyperref}
\allowdisplaybreaks

\newtheorem{thm}{Theorem}[section]
\newtheorem{lem}[thm]{Lemma}
\newtheorem{cor}[thm]{Corollary}

\newtheorem{defn}[thm]{Definition}
\newtheorem{rmk}{Remark}

\numberwithin{equation}{section}
\newcommand{\DDD}{\mathcal{D}}

\newcommand{\ep}{\varepsilon}
\newcommand{\la}{\lambda}
\newcommand{\va}{\varphi}
\newcommand{\ppp}{\partial}

\newcommand{\whwh}{\widehat}
\newcommand{\uull}{u_{\lambda}}

\newcommand{\ddda}{d_t^{\alpha}}

\newcommand{\MAAA}{\mathcal{A}}

\newcommand{\pppa}{\partial_t^{\alpha}}

\newcommand{\R}{\mathbb{R}}
\newcommand{\C}{\mathbb{C}}
\newcommand{\N}{\mathbb{N}}

\newcommand{\www}{\widetilde}

\newcommand{\ooo}{\overline}
\newcommand{\OOO}{\Omega}

\newcommand{\LTLT}{L^2(0,T;X)}
\newcommand{\LTLTX}{L^2(0,T;X)}

\newcommand{\JJJJ}{\mathcal{J}_{\lambda}}
\newcommand{\AAAA}{A_{\lambda}}

\allowdisplaybreaks

\title
[Well-posedness for time-fractional evolution equations]
{
Well-posedness for time-fractional evolution equations by the 
Yosida approximation
\vspace{1cm}
}

\author{
$^1$ Salah-Eddine Chorfi, $^2$ Fikret G\"olgeleyen, $^{3,2}$ Masahiro Yamamoto
\vspace{1cm}
}
\thanks{\hspace{-0.4cm}$^1$ Cadi Ayyad University, UCA, Faculty of Sciences Semlalia, Laboratory of Mathematics, Modeling and Automatic Systems, B.P. 2390, Marrakesh, Morocco, e-mail: {\tt s.chorfi@uca.ac.ma}\\
$^2$ Department of Mathematics, Faculty of Science, 
Zonguldak B\"ulent Ecevit University, Zonguldak 67100 T\"urkiye,
e-mail: {\tt f.golgeleyen@beun.edu.tr}
\\
$^3$ Graduate School of Mathematical Sciences, The University
of Tokyo, Komaba, Meguro, Tokyo 153-8914, Japan,
e-mail: {\tt myama@ms.u-tokyo.ac.jp}}

\date{}
\begin{document}

\begin{abstract}
We consider an initial value problem for a time-fractional evolution equation
in a Hilbert space $X$:
$$
\partial_t^{\alpha} (u(t)-a) = Au(t) \qquad \mbox{for $0<t<T$},
$$
where $\partial_t^{\alpha}$ denotes a fractional differential operator of Caputo type with order $0<\alpha<1,$ $u\colon (0,T) \to X,$ $a$ describes an initial value, and $A$, with domain $\mathcal{D}(A),$ is the generator of a contraction C$_0$ semigroup in $X$.
We prove a fractional Hille-Yosida theorem characterizing the unique existence of a weak solution for every $a\in X$. We also discuss the unique existence of a strong solution for $a \in \mathcal{D}(A),$ as well as the continuity of weak solutions. Our results are applicable, for example, to time-fractional transport equations and Boltzmann equations. The proofs are mainly based on constructing an approximating sequence of solutions using the Yosida approximation.
\end{abstract} 



\maketitle

\baselineskip 18pt

\section{Introduction}
Fractional evolution equations have attracted considerable attention in recent years because they can describe dynamical processes with memory and hereditary effects. In contrast to classical evolution equations, where the instantaneous rate of change of the state depends only on its current configuration, time-fractional models account for the influence of the past through a nonlocal operator in time. Such equations arise naturally in anomalous diffusion, viscoelasticity, transport in heterogeneous media, relaxation phenomena, and many other applications; see, e.g., \cite{CM71, MG00, MK00, SB03}.
\smallskip

In this paper, we are concerned with the well-posedness of abstract 
time-fractional evolution equations. More precisely, let $(X,(\cdot,\cdot)_X)$ 
be a real Hilbert space and $0 < \alpha < 1$. We can also consider a complex 
Hilbert space with simple adaptations.
We consider an evolution equation in $X$:
\begin{equation}\label{(1.1)}
\pppa (u(t)-a) = Au(t)\quad \mbox{in $X$ for $0<t<T$},
\end{equation}
where $\pppa$ is a fractional derivative of order $\alpha$ that will be 
defined later.

From a mathematical perspective, the presence of the time-fractional derivative
in \eqref{(1.1)} leads to substantial differences from the classical theory of 
evolution equations ($\alpha=1$) and requires analytical tools capable of 
handling its nonlocal-in-time structure. One of the main challenges is therefore to determine to what extent the powerful methods developed for classical evolution equations can be adapted to the fractional setting. We refer, for instance, to the monographs \cite{J, KRY, Po} for systematic studies of fractional differential equations and some related well-posedness results. We also emphasize that the reference list for time-fractional evolution equations is not exhaustive, and we mention only a few.

Our purpose is to develop a direct well-posedness approach based on the Yosida approximation of the operator $A$. Yosida approximations have previously been employed in the fractional setting; for instance, Kamenskii et al.~\cite{Ka17} used them to justify approximation procedures for semilinear fractional equations. In contrast, our approach uses the Yosida approximation as a direct tool for constructing both weak and strong solutions. We mention the work by Bazhlekova 
\cite{Baz98}, where solvability criteria for fractional evolution equations in Banach spaces are established using subordination and Laplace-transform techniques. We also refer to the classical book by Pr\"uss \cite{Pru93} for a broad formulation by integral equations. In comparison, our formulation is based on a direct realization of the fractional differential operator, which is particularly well suited to the $L^2$-Hilbert space framework considered here. This allows us to obtain well-posedness by direct energy estimates and weak compactness arguments under relaxed regularity assumptions. Furthermore, the integral equation formulation may conflate the initial datum and the source term, a distinction that is essential in many applications, particularly in inverse problems. It may also lead to difficulties when treating certain ranges of the order $\alpha.$

Throughout this article, unless otherwise stated, we assume:
\begin{equation}\label{(1.2)}
\begin{cases}
& \mbox{$A$ is a closed densely defined operator in $X$, and}\\
& {\bf Condition\, (\mathcal{M}):}\\
&\{ \la \in \R; \, \la> 0\} \subset \rho(A) \, 
\mbox{: the resolvent set of $A$}, \\
& \la\Vert (\la-A)^{-1}\Vert_{X\to X} \le 1 \quad \mbox{for all $\la > 0$}.
\end{cases}
\end{equation}

The work \cite{Baz98} shows that for more general $A$, some estimates of all the orders of derivatives of the resolvent imply the well-posedness of the initial value problem, and technically such estimates correspond to the convergence of a related Taylor series, but the verification of the estimates for all the orders of derivatives is not simple. Condition $(\mathcal{M})$ restricts such estimates of derivatives of all orders to only the first-order derivative within a smaller class of $A$. However, our class of $A$ covers
transport equations, as Examples 1 and 2 below show. 
 
Assumption \eqref{(1.2)} is a necessary and sufficient condition that $A$ generates a contraction C$_0$ semigroup 
$e^{tA}$ for $t\ge 0$ (e.g., Pazy \cite{Pa}, Tanabe \cite{Ta}).

The results obtained in this paper are formulated at an abstract level and therefore apply to a broad class of partial differential equations whose spatial part is realized as the generator of a C$_0$ semigroup. Relevant examples include time-fractional transport equations, for which relatively few well-posedness results are available, partly because the standard method of characteristics does not extend directly to the fractional setting; see, for instance, \cite{ER18} and \cite{GN17} for some fractional transport equations. Our approach also applies to time-fractional Schrödinger equations, which have recently attracted increasing attention; we refer to \cite{CEMY26} for related results.

Henceforth, for an operator $K$ and Banach spaces $Y, Z$, we mean by $K: Y \rightarrow Z$ that the domain $\DDD(K)$ is 
included in $Y$ with values in $Z$, but is not necessarily defined over the whole space $Y.$

The paper is organized as follows. In Section \ref{sec2}, we present the main well-posedness results. The main results are proved in Sections \ref{sec3} - \ref{sec7}. Section \ref{sec8} is devoted to concluding remarks, and Section \ref{sec9} provides the proofs of necessary lemmata used throughout the paper.

\section{Main results}\label{sec2}

We introduce function spaces for the Hilbert space $X$ with 
the norm $\Vert \cdot\Vert_X$ and the scalar product 
$(\cdot,\cdot)_X$.
First we set
\begin{equation}\label{(2.1)}
L^2(0,T;X):= \left\{ v:(0,T) \longrightarrow X \text{ strongly measurable};\,
\int^T_0 \Vert v(t)\Vert^2_X dt < \infty\right\}
\end{equation}
with $(v,w)_{L^2(0,T;X)}:= \int^T_0 (v(t),w(t))_X dt$
and $\Vert v\Vert_{L^2(0,T;X)} := ((v,v)_{L^2(0,T;X)})^{\frac{1}{2}}$.

We define the graph norms $\Vert \cdot\Vert_{\DDD(A)}$ and 
$\Vert \cdot\Vert_{\DDD(A^*)}$ by 
$$
\Vert u\Vert_{\DDD(A)}: = \Vert u\Vert_X + \Vert Au\Vert_X
\quad \mbox{for $u\in \DDD(A)$},
$$
and
$$
\Vert u\Vert_{\DDD(A^*)}: = \Vert u\Vert_X + \Vert A^*u\Vert_X
\quad \mbox{for $u\in \DDD(A^*)$}.
$$

Then, we define an operator $\MAAA: L^2(0,T;\DDD(A)) \rightarrow \LTLTX$ by 
$$
(\MAAA u)(t) = Au(t) \quad \mbox{for $u\in \DDD(\MAAA)
:= L^2(0,T;\DDD(A))$}.
$$
For simplicity, sometimes we write $Au$ to mean $\MAAA u$, etc. We can easily prove that $\ooo{\DDD(\MAAA)^{\LTLTX}} = \LTLTX$.
Hence, we can define the adjoint operator $\MAAA^*$ to 
$\MAAA : \LTLTX \rightarrow \LTLTX$.  More precisely, we can 
specify $\MAAA^*$ as the maximal operator among operators $\mathcal{B}: \LTLTX
\rightarrow \LTLTX$ satisfying
$(\MAAA u, \va)_{\LTLTX} = (u, \mathcal{B}\va)_{\LTLTX}$ for all $u \in \DDD(\MAAA)$.
In other words,
\begin{align*}
& \DDD(\MAAA^*) = \{ \va\in \LTLTX;\, \mbox{we can find $w\in 
\LTLTX$ satisfying}\\
& (\MAAA u,\va)_{\LTLTX} = (u,w)_{\LTLTX} \,\, \mbox{for all 
$u\in \DDD(\MAAA)$} \},
\end{align*}
and we define $\MAAA^*\va := w$.

\begin{rmk}
Here and henceforth, we define the adjoint operators with respect to the scalar product in $\LTLTX$.  We remark that we do not consider 
adjoint operators with duality product $_{Y^*}\langle\cdot, \, \cdot\rangle_{Y}$ where $X$ is a pivot space and $Y \hookrightarrow X \hookrightarrow Y^*$ is a prescribed Gel'fand triple (see e.g., Brezis \cite{Bre}, Section 6 of Chapter 1 in Yagi \cite{Ya}). Through such a duality product, in Yamamoto \cite{Y22}, we can 
define $(\pppa)^*$ in Sobolev spaces of negative orders, but in this article we do not consider adjoint operators in this way.
\end{rmk}
 
We set 
$$
J^{\alpha}v(t) := \frac{1}{\Gamma(\alpha)}
\int^t_0 (t-s)^{\alpha-1} v(s) ds \quad \mbox{for $v\in L^2(0,T;X)$}.
$$
The operator $J^\alpha$ is bounded and injective on $L^2(0,T;X)$. Note that we can also define $J^\alpha v$ for $v\in L^1(0,T;X)$. Then, we define $\pppa = (J^{\alpha})^{-1}$ with 
$\DDD(\pppa) = J^{\alpha}\LTLTX=:H_{\alpha}(0,T;X)$ and the norm 
$\|v\|_{H_{\alpha}(0,T;X)}:=\|\pppa v\|_{\LTLTX}$.
We can prove that $\ooo{\DDD(\pppa)^{\LTLTX}} = \LTLT$, and so
we can define the adjoint operator $(\pppa)^*$ to $\pppa$.

Now we consider
\begin{equation}\label{(2.2)}
\pppa (u(t)-a) = Au(t) \quad \mbox{in $X$ for $0<t<T$},
\end{equation}
and introduce
\begin{defn}[weak solution]\label{defws}
We call $u=u(t)$ a weak solution to \eqref{(2.2)} if $u\in L^2(0,T;X)$ 
satisfies
\begin{equation}\label{(2.3)}
(u-a,\, (\pppa)^*\va)_{L^2(0,T;X)}
= (u,\, \MAAA^*\va)_{\LTLT}
\end{equation}
for all $\va \in \DDD((\pppa)^*) \cap \DDD(\MAAA^*)$.
\end{defn}

\begin{defn}[strong solution]\label{defss}
We call $u=u(t)$ a strong solution to \eqref{(2.2)} if
\begin{equation*}
u\in L^2(0,T;\DDD(A)), \qquad u-a\in H_\alpha(0,T;X),
\end{equation*}
and \eqref{(2.2)} is satisfied in $L^2(0,T;X)$.
\end{defn}



We will prove the following main results.
\begin{thm}[time-fractional Hille-Yosida theorem]\label{thm1}
We assume that $A$ is a densely defined closed operator. In order that for each $a \in X$, there exists a unique weak solution to \eqref{(2.2)} satisfying
$u \in C([0,T];X)$ and  
\begin{equation}\label{(2.4)}
\Vert u(t)\Vert_X \le \Vert a\Vert_X \quad \text{ for all } t \ge 0,
\end{equation}
Condition $(\mathcal{M})$ is necessary and sufficient.
\end{thm}

\begin{rmk}
This theorem is a generalization to $0<\alpha<1$ of the classical Hille-Yosida theorem for $\alpha=1,$ which asserts the equivalence of Condition $(\mathcal{M})$ and the well-posedness of $u'(t) = Au$ with $u(0) = a\in X,$ where $A$ generates a contraction C$_0$ semigroup (e.g., Pazy \cite{Pa} and Tanabe \cite{Ta}). We refer to Da Prato and Iannelli \cite{Da80} and Sforza \cite{Sf86} 
for some related results for integro-differential equations.
\end{rmk}

Next, we prove that any weak solution automatically belongs to 
$C([0,T];X)$.
\begin{thm}\label{thmcont}
We assume \eqref{(1.2)}.
If $u\in L^2(0,T;X)$ is a weak solution of \eqref{(2.2)}
for $a \in X$, then
$$
u\in C([0,T];X), \qquad u(0)=a.
$$
\end{thm}

Moreover, we can prove an improved regularity result.
\begin{thm}\label{cor1.4}
We assume that there exists a unique weak solution $u$ to $\pppa (u-a) = Au$ in $X$ for each $a\in X$.
\begin{itemize}
    \item[(i)] For $a \in \DDD(A)$, the weak solution $u$ satisfies $u-a \in H_\alpha(0,T;X)$, $u\in L^2(0,T;\DDD(A))$ and
$\pppa (u(t)-a) = Au(t)$ in $X$ for almost all $t\in (0,T)$.
    \item[(ii)] Let $a \in \DDD(A^m)$ with $m\in \N$. Then the weak solution $u$ satisfies $A^mu \in L^2(0,T;X)$. 
\end{itemize}
\end{thm}
We emphasize that this improvement is proved independently of Condition $(\mathcal{M})$. Theorem \ref{cor1.4} readily implies
\begin{cor}[unique existence of strong solution]\label{cor1}
Let $a \in \DDD(A)$. Then there exists a unique strong solution to \eqref{(2.2)} satisfying \eqref{(2.4)} and 
\begin{equation}\label{eq*}
u, \, Au \in C([0,T];X), \quad
\Vert Au(t)\Vert_X \le \Vert Aa\Vert_X \quad \mbox{for all $t \in [0,T]$.}
\end{equation}
\end{cor}
We remark that the estimate \eqref{eq*} follows from \eqref{(2.4)} in Theorem \ref{thm1}, since 
$v:= Au$ is a weak solution to $\pppa (v-Aa) = Av$.

Now, we provide bounded perturbation results to broaden the applicability of our approach. Let $B$ be a bounded operator on $X$ and consider the perturbed problem
\begin{equation}\label{(pert)}
\pppa (u(t)-a) = Au(t)+Bu(t)\quad \mbox{for $0<t<T$},
\end{equation}
where the weak and strong solutions are defined by replacing $A$ by $A+B$ in Definition \ref{defws} and Definition \ref{defss}.

\begin{thm}\label{thmpert}
Let us assume \eqref{(1.2)} and let $a \in X$.  Then there exists a unique weak solution 
$u\in L^2(0,T;X)$ to \eqref{(pert)} such that
$$
\Vert u\Vert_{L^2(0,T;X)} \le C\Vert a\Vert_X \quad \mbox{for almost all 
$t \in (0,T)$},
$$
where the constant $C>0$ depends only on $\alpha, T,$ and $\|B\|_{X\to X}$.
\end{thm}


Thus, we can summarize our main results under assumption \eqref{(1.2)} as follows:
\\
{\bf Theorem 2.8.}
{\it
We assume that \eqref{(1.2)} holds and $B: X \rightarrow X$
is a bounded linear operator on $X$.
\\
(i) For each $a\in X$, there exists a unique weak solution 
$u \in L^2(0,T;X)$ to \eqref{(pert)} and
we can choose a constant $C>0$ such that
$$
\Vert u(t)\Vert_X \le C\Vert a\Vert_X \quad \mbox{for all $a\in X$
and almost all $t \in (0,T)$}.
$$
(ii) For each $a\in \DDD(A)$, there exists a unique strong solution to 
\eqref{(pert)} satisfying $u,\, Au\in C([0,T];X)$ and
we can choose a constant $C>0$ such that 
$$
\Vert Au(t)\Vert_X + \Vert u(t)\Vert_X 
\le C(\Vert Aa\Vert_X + \Vert a\Vert_X) \quad \mbox{for all 
$a \in \DDD(A)$ and almost all $t \in (0,T)$}.
$$
}

Now we give important examples of operators satisfying Condition 
$(\mathcal{M})$.

Throughout the sequel, let $\OOO \subset \R^d$ be a bounded domain with boundary $\ppp\OOO$ of class $C^1$ and $\nu = \nu(x)$ be the outward unit normal vector to $\ppp\OOO$ at $x\in \ppp\OOO$.
\smallskip

\noindent{\bf Example 1 (fractional transport equation).}

Let $p\in L^\infty(\Omega)$ and $h := (h_1, ..., h_d) \in (C^1(\ooo{\OOO}))^d$.
We set
$$
\left\{\begin{array}{rl}
& \ppp\OOO_+:= \{ x \in \ppp\OOO;\, h(x) \cdot \nu(x) > 0\}, \\
& \ppp\OOO_-:= \{ x \in \ppp\OOO;\, h(x) \cdot \nu(x) < 0\}.
\end{array}\right.
$$
We define an operator $A$ by
$$
\left\{ \begin{array}{rl}
& Au := -h(x)\cdot \nabla u + p(x) u \quad \mbox{in $\OOO$},  \\
& \DDD(A):= \{ u\in L^2(\OOO);\, A u \in L^2(\OOO), \, 
u\vert_{\ppp\OOO_-} = 0\}.
\end{array}\right.
$$
We understand $A u$ in the sense of distributions.
If $u,\, A u \in L^2(\OOO)$, then $u\vert_{\ppp\OOO_-}$ can be defined in a suitable sense on 
$\ppp\OOO_-$ (e.g., Bardos \cite{Ba1}, \cite{Ba2}, \cite{Ba3}). 

Next, we define the operator
$$
\widetilde{A} u:=-h(x) \cdot \nabla u-\frac{1}{2}(\operatorname{div} h(x)) u,
$$
with domain
$$
\DDD(\widetilde{A})=\left\{u \in L^2(\Omega);\, h \cdot \nabla u \in 
L^2(\Omega),\left.u\right|_{\partial \Omega_-}=0\right\} .
$$
The operator $\widetilde{A}$ is closed and densely defined. Theorem 2.3 (p.~207) in \cite{Ba3} implies that there exists $\la_0 > 0$ sufficiently large such that $\la_0 - \widetilde{A}$ is surjective. Moreover, by integration by parts (see Proposition 2.5 in \cite{Ba3}), for all $u\in \DDD(\widetilde{A})$,
$$
\begin{aligned}
(\widetilde{A} u, u)_{L^2(\Omega)} & =(-h \cdot \nabla u, u)_{L^2(\Omega)}-\frac{1}{2} \int_{\Omega}(\operatorname{div} h)|u|^2 d x \\
& =-\frac{1}{2} \int_{\partial \Omega_+}(h \cdot \nu)|u|^2 d \sigma \leq 0.
\end{aligned}
$$
By Proposition 2.1.4 in Tanabe \cite{Ta}, we see that $\{ \la \in \C;\, \mbox{Re}\, \la > 0\} \subset \rho(\widetilde{A})$. Then, \cite[Theorem 2.1.3]{Ta} yields that $\widetilde{A}$ satisfies Condition $(\mathcal{M})$. Thus, $\widetilde{A}$ generates a contraction C$_0$ semigroup in $L^2(\OOO)$. We can write $A=\widetilde{A} + B,$ where $B$ is defined on $L^2(\Omega)$ by 
$$
Bu (x) := (p(x)+\frac{1}{2}\operatorname{div} h(x)) u(x).
$$ 
Since $p +\frac{1}{2}\operatorname{div} h \in L^\infty(\Omega)$, the operator $B$ is bounded. Therefore, by Theorem \ref{thmpert}, the fractional transport equation governed by the operator $A$ admits a unique weak solution $u\in L^2(0,T; L^2(\OOO)).$
\smallskip

\noindent{\bf Example 2 (fractional Boltzmann equation).}

Let $U$ be an open set given by $$U:= \{ \xi\in \R^d;\, \chi_0 < \vert \xi\vert 
< \chi_1\}$$
with constants $0<\chi_0 < \chi_1$.
We set 
$$
\begin{cases}
& \Gamma_+:= \{ (x,\xi)\in \ppp\OOO \times U;\, 
\nu(x)\cdot \xi > 0\}, \\
& \Gamma_-:= \{ (x,\xi)\in \ppp\OOO \times U;\, 
\nu(x)\cdot \xi < 0\}.
\end{cases}
$$
Then, we define
\begin{align*}
A_0 u(x,\xi) &:= -\xi\cdot \nabla_x u(x,\xi), \quad (x,\xi)\in \OOO
\times U, \\
A u(x,\xi) &:= -\xi\cdot \nabla_x u(x,\xi) + p(x,\xi)u(x,\xi) + \int_U k(x,\xi,\xi')u(x,\xi') d\xi' ,
\end{align*}
where $p\in L^\infty(\OOO\times U)$ and $k\in L^\infty(\Omega;L^2(U\times U)).$ The domain is given by
$$
\DDD(A)=\DDD(A_0): = \{ u\in L^2(\OOO\times U),\,
A_0 u \in L^2(\OOO\times U) \quad \mbox{and}\quad
u\vert_{\Gamma_-} = 0\}.
$$
Here, we understand $A_0 u$ in the sense of distributions. Note that the trace $u\vert_{\Gamma_-}$ is well-defined in this case; see Dautray and Lions \cite{DaLi}, p.~222. It is known that $A_0$ satisfies \eqref{(1.2)} by \cite[Theorem 2]{DaLi}. Moreover, the operator defined by
$$
B u(x,\xi)=p(x,\xi) u(x,\xi)+\int_U k\left(x, \xi, \xi^{\prime}\right) u\left(x, \xi^{\prime}\right) d \xi^{\prime}
$$
is bounded on $L^2(\Omega \times U)$. Therefore, by Theorem \ref{thmpert}, the fractional Boltzmann equation governed by the operator $A$ admits a unique weak solution $u\in L^2(0,T; L^2(\OOO\times U)).$
\smallskip

Next, we collect the properties of the spaces, which are used for the 
proofs.
\begin{lem}\label{lem3.1}
(i) The operators $\MAAA: \LTLTX \longrightarrow \LTLTX$ with $\DDD(\MAAA)
= L^2(0,T;\DDD(A))$ and 
$\pppa: \LTLTX \longrightarrow \LTLTX$ are closed.
\\
(ii) $\DDD(\MAAA^*) = L^2(0,T;\DDD(A^*))$.
\\
(iii) $C^{\infty}_c(0,T;X) \subset \DDD((\pppa)^*)$.
\\
(iv) The space $\DDD(\MAAA^*) \cap \DDD((\pppa)^*)$ is dense in 
$\LTLTX$.
\\
(v) The space $\DDD(\MAAA^*) \cap \DDD((\pppa)^*)$ is dense in 
$\DDD(\MAAA^*)$ with the graph norm.
\end{lem}

Correspondingly to $J^{\alpha}$, we define 
$$
J_{\alpha}v(t) := \frac{1}{\Gamma(\alpha)}
\int^T_t (\xi-t)^{\alpha-1} v(\xi) d\xi \quad \mbox{for $v\in L^2(0,T;X)$}.
$$
Then, $J_\alpha$ is a bounded injective operator on $L^2(0,T;X)$, and we have:
\begin{lem}\label{lem3.2}
(i) $J_{\alpha}L^2(0,T;X)$ is dense in $L^2(0,T;X)$.
\\
(ii) $(\pppa)^* = (J_{\alpha})^{-1}$ and $\DDD((\pppa)^*) 
= J_{\alpha}L^2(0,T;X)$.
\\
(iii) $(J^{\alpha}f,\, g)_{L^2(0,T;X)}
= (f, \, J_{\alpha}g)_{L^2(0,T;X)}$ for all $f,g \in L^2(0,T;X)$.
\end{lem}
The proofs of Lemmata \ref{lem3.1} and \ref{lem3.2} are sketched in 
Section \ref{sec9}.

Sections \ref{sec5} - \ref{sec7} are devoted to the proofs of Theorems \ref{thm1}, \ref{cor1.4} and \ref{thmpert}, respectively.
As for the main results in the case $a \in X$, the proofs are organized as follows.
\begin{itemize}
\item
Section \ref{sec3} proves the uniqueness of weak solution;
\item
Section \ref{sec4} proves that Condition $(\mathcal{M})$ is equivalent to 
the existence of weak solution with a contraction estimate;
\item
Section \ref{sec5} proves that if $u$ is weak solution, then  
$u\in C([0,T];X)$.
\end{itemize}
 
\section{Proof of the uniqueness of weak solutions} \label{sec3}

In this section, we will prove
\begin{lem}\label{llem3.1}
We assume \eqref{(1.2)}. Let $u\in L^2(0,T;X)$ be a weak solution to $\pppa u = Au$ in $X$ 
for $0<t<T$.  Then, $u=0$.
\end{lem}

\noindent{\bf Proof of Lemma \ref{llem3.1}.}\\
First, we assume that $u \in L^2(0,T;\DDD(A))\cap \DDD(\pppa)$ is a weak solution to $\pppa u = Au$. Then, we have $(\pppa u, u)_{\LTLTX} \ge \frac{T^{-\alpha}}{2\Gamma(1-\alpha)}\Vert u\Vert^2_{\LTLTX}$ (e.g., Theorem 3.3 (ii) in \cite{KRY}).
Since $(Au(t),\, u(t))_X \le 0$ for almost all $t \in (0,T)$ by Lemma \ref{lem9.1}, we obtain $0 \ge \frac{T^{-\alpha}}{2\Gamma(1-\alpha)}\Vert u\Vert^2_{\LTLTX},$ and so $u=0$ in $\LTLTX$. Thus the uniqueness is proved within $L^2(0,T;\DDD(A)) \cap \DDD(\pppa)$.
 
Next, let $w\in L^2(0,T;X)$ satisfy 
\begin{equation}\label{(4.1)}
(w, \, (\partial_t^\alpha)^* \varphi)_{L^{2}(0,T;X)}
= (w, \mathcal{A}^* \varphi)_{L^{2}(0,T;X)}
\end{equation}
for all $\varphi \in \DDD((\partial_t^\alpha)^*) \cap L^2(0,T;\DDD(A^*))$. We need to prove that $w=0$.

Fix $\lambda>0$ and define $z:=(\lambda -A)^{-1} w$. Then,
\begin{equation}\label{(4.2)}
z\in L^2(0,T;\DDD(A)), \quad Az=\lambda z-w \quad\text{in }
L^2(0,T;X).
\end{equation}
We will show that
\begin{equation}\label{(4.3)}
z\in \DDD(\partial_t^\alpha) \quad\text{and}\quad 
\partial_t^\alpha z=Az.
\end{equation}
Let $\psi\in \DDD((\partial_t^\alpha)^*)$.
Since $\DDD((\partial_t^\alpha)^*) = J_{\alpha}L^2(0,T;X)$ by 
Lemma \ref{lem3.2} (ii), 
there exists $g\in L^2(0,T;X)$ such that $\psi=J_{\alpha}g$.
Because $(\la-A^*)^{-1}$ is bounded on $X$ and acts only in $X$, 
it commutes with the integral defining $J_{\alpha}$. Hence,
$$
(\la-A^*)^{-1}\psi = (\la-A^*)^{-1}J_{\alpha}g = J_{\alpha}
(\la-A^*)^{-1}g.
$$
Consequently, $(\la-A^*)^{-1}\psi \in \DDD((\partial_t^\alpha)^*)$ and
\begin{equation}\label{(4.4)}
(\partial_t^\alpha)^*(\la-A^*)^{-1}\psi 
= (\la-A^*)^{-1}(\partial_t^\alpha)^*\psi.
\end{equation}
Moreover, $(\la-A^*)^{-1}\psi\in L^2(0,T;\DDD(A^*)),$ so that 
$(\la-A^*)^{-1}\psi$ is an admissible test function in \eqref{(4.1)}.

Using \eqref{(4.1)} and \eqref{(4.4)}, we obtain
\begin{align*}
& (z,\, (\partial_t^\alpha)^*\psi)_{L^2(0,T;X)} 
= ((\la-A)^{-1}w,\, (\partial_t^\alpha)^*\psi)_{L^2(0,T;X)} \\
= & (w,\, (\la-A^*)^{-1}(\partial_t^\alpha)^*\psi)_{L^2(0,T;X)} 
  = (w,\, (\partial_t^\alpha)^* (\la-A^*)^{-1}\psi)
_{L^2(0,T;X)}                                     \\ 
= & (w,\, A^*(\la-A^*)^{-1}\psi)_{L^2(0,T;X)}.
\end{align*}
For the last equality, we used \eqref{(4.1)}, that is, the fact that 
$w$ is a weak solution.
Hence, since $((\la-A^*)^{-1})^* = (\la - A)^{-1}$, we have 
\begin{align*}
& (z,\, (\partial_t^\alpha)^*\psi)_{L^2(0,T;X)} 
= (w,\, \la(\la-A^*)^{-1}\psi - \psi)_{L^2(0,T;X)} \\
= & (\lambda (\la - A)^{-1}w - w,\, \psi)_{L^2(0,T;X)} 
= (\lambda z-w,\, \psi)_{L^2(0,T;X)} = (Az,\, \psi)_{L^2(0,T;X)}.
\end{align*}
That is, we obtain
$$
(z, \, (\pppa)^*\psi)_{\LTLTX} = (Az,\, \psi)_{\LTLTX}
\quad \mbox{for all $\psi \in \DDD((\pppa)^*)$}.
$$
Thus, by the definition of the adjoint operator, we have shown that $z\in \DDD(((\pppa)^*)^*)$ and $((\pppa)^*)^*z=Az$. Since $\pppa$ is densely defined and closed, $((\pppa)^*)^*=\pppa$. Therefore, we have proved \eqref{(4.3)}.
By the uniqueness of the strong solution, which we already proved, 
we obtain $z=0$. Since $(\la - A)^{-1}$ is injective, we reach $w=0$. 
This proves the uniqueness of the weak solution.
$\blacksquare$

\section{Proof of Theorem \ref{thm1}}\label{sec4}

\noindent{\bf 4.1. Proof of the sufficiency.}

Thanks to Lemma \ref{llem3.1}, it suffices to 
prove that Condition $(\mathcal{M})$ implies the 
existence of a weak solution which belongs to $C([0,T];X)$.

\noindent{\bf First Step: construction of approximating solutions.}
\\
Now, we construct a sequence approximating a weak solution. 
Let $\la > 0$ be a parameter. We set 
$$
\JJJJ:= \la(\la-A)^{-1}, \quad 
\AAAA:= \la A(\la-A)^{-1} = -\la + \la\JJJJ.
$$
The operators $\JJJJ$ and $\AAAA$ are bounded and defined on $X$, and $\AAAA$ is called the Yosida approximation.
Then, the following is known (\cite{Pa}, \cite{Ta}, for example).
\begin{lem}\label{lem5.1}
Let us assume that $A$ satisfies \eqref{(1.2)}. Then,
\\
(i) $\lim_{\la\to\infty} \JJJJ b = b$ in $X$ for all $b\in X$.
\\
(ii) $\lim_{\la\to\infty} \AAAA b = Ab$ in $X$ for all $b\in \DDD(A)$.
\\
(iii)  $\lim_{\la\to\infty} \AAAA^* b = A^*b$ in $X$ for all 
$b\in \DDD(A^*)$.
\\
(iv) $\Vert \JJJJ \Vert_{X\to X} \le 1$ for all $\la > 0$.
\\
(v) $(A_{\la}u,u)_X \le 0$ for all $\la > 0$ and $u\in X$.
\\
(vi) $\DDD(A^2)$ is dense in $X$.
\end{lem}
The proofs can be found, e.g., in \cite{Pa} and \cite{Ta}.
For convenience, in Section \ref{sec9}, 
we provide the proof of $(v)$. The assertion $(vi)$ follows from 
\cite[Theorem 2.7]{Pa}.

Here we recall the Caputo derivative defined by
$$
\ddda v(t) = \frac{1}{\Gamma(1-\alpha)}\int^t_0 
(t-s)^{-\alpha} v'(s) ds \quad 
\mbox{for $v \in W^{1,1}(0,T;X)$.}
$$
Moreover, we can prove
\begin{lem}\label{lem5.2}
$$
(J^{\alpha}(\ddda v, v)_X)(t) \ge \frac{1}{2}\Vert v(t)\Vert_X^2
- \frac{1}{2}\Vert v(0)\Vert^2_X, \quad 0<t<T
$$
for $v \in W^{1,1}(0,T;X)$.
\end{lem}
The proof can be found, e.g., in Theorem 3.2 (i) in \cite{KRY}.
\smallskip

Now, we construct weak solutions via the Yosida approximation. 
For $\la>0$, we consider an approximating system:
\begin{equation}\label{(5.1)}
\pppa (u_\la-a) = \AAAA u_\la \quad \mbox{in $X$ for $0<t<T$.}
\end{equation}
Then, we prove
\begin{lem}\label{lem5.3}
For each $\la > 0$ and $a \in X$, there exists a unique solution 
$u_{\la}$ to \eqref{(5.1)}  such that $u_{\la}-a \in H_{\alpha}(0,T;X)$.
Moreover, 
$$
u_{\la} \in W^{1,1}(0,T;X), \quad u_{\la}(0) = a.
$$
\end{lem}
\noindent{\bf Proof of Lemma \ref{lem5.3}.}
\\
To construct $u_\la,$ we write \eqref{(5.1)} as
\begin{align*}
& u_{\la}(t) = a + \int^t_0 \frac{1}{\Gamma(\alpha)}
(t-s)^{\alpha-1} \AAAA u_{\la}(s) ds   \\
&= a + \int^t_0 \frac{1}{\Gamma(\alpha)}
s^{\alpha-1} \AAAA u_{\la}(t-s) ds, \quad 0<t<T.
\end{align*}
Henceforth, we denote
$$
u_{\la}'(t) := \frac{du_{\la}}{dt}(t), \quad \mbox{etc.}
$$
Formally, we differentiate both sides with respect to 
$t$, so that we can obtain a candidate equation for 
$u_{\la}'$:
$$
u_{\la}'(t) = \frac{1}{\Gamma(\alpha)}t^{\alpha-1} \AAAA a
+ \int^t_0 \frac{1}{\Gamma(\alpha)}
  s^{\alpha-1} \AAAA u_{\la}'(t-s) ds.
$$
Thus, to prove the existence of $u_{\la}' \in L^1(0,T;X)$,
we introduce the following iteration:
$$
\left\{ \begin{array}{rl}
& w_0(t) := 0, \quad w_1(t) := \frac{1}{\Gamma(\alpha)}t^{\alpha-1} \AAAA a, \\
& w_{n+1}(t) = \frac{1}{\Gamma(\alpha)}t^{\alpha-1} \AAAA a
+ \int^t_0 \frac{1}{\Gamma(\alpha)} s^{\alpha-1} \AAAA w_{n}(t-s) ds,
\quad n\in \N.
\end{array}\right.
$$
We note that $\AAAA: X \longrightarrow X$ is a bounded operator and 
$\Vert \AAAA\Vert_{X\to X} \le 2\la$.

We can estimate
\begin{align*}
& \Vert (w_{n+2} - w_{n+1})(t)\Vert 
\le \frac{1}{\Gamma(\alpha)}\int^t_0 s^{\alpha-1} \Vert \AAAA\Vert_{X\to X}
\Vert (w_{n+1} - w_n)(t-s)\Vert ds \\
\le & C\int^t_0 (t-s)^{\alpha-1} 
\Vert (w_{n+1} - w_n)(s)\Vert ds.
\end{align*}
Here $C>0$ depends on $\la$.

First,
$$
\Vert (w_1-w_0)(s)\Vert_X = \Vert w_1(s)\Vert_X
= \left\Vert \frac{1}{\Gamma(\alpha)}(\AAAA a)s^{\alpha-1} 
\right\Vert_X \le Ms^{\alpha-1},
$$ 
where we set $M:=\frac{1}{\Gamma(\alpha)}\Vert \AAAA a\Vert_X$.
Next,
$$
\Vert (w_2-w_1)(t)\Vert_X 
\le CM\int^t_0 (t-s)^{\alpha-1}s^{\alpha-1} ds
= CM\frac{\Gamma(\alpha)^2}{\Gamma(2\alpha)}t^{2\alpha-1},
$$
and
$$
\Vert (w_3-w_2)(t)\Vert_X 
\le C^2M\frac{\Gamma(\alpha)^2}{\Gamma(2\alpha)}
\int^t_0 (t-s)^{\alpha-1}s^{2\alpha-1} ds
= \frac{C^2M\Gamma(\alpha)^3}{\Gamma(3\alpha)}t^{3\alpha-1}
$$
for $0<t<T$.
Continuing the estimation, we obtain
$$
\Vert (w_{n+1} - w_n)(t)\Vert_X 
\le \frac{M\Gamma(\alpha)(C\Gamma(\alpha))^n}
{\Gamma((n+1)\alpha)}t^{(n+1)\alpha-1},
$$
and so 
$$
\Vert w_{n+1} - w_n\Vert_{L^1(0,T;X)} 
\le \frac{M\Gamma(\alpha)(C\Gamma(\alpha))^nT^{(n+1)\alpha}}
{\Gamma((n+1)\alpha+1)}
\le C_0\frac{C_1^n}{\Gamma((n+1)\alpha+1)}
$$
for all $n\in \N$, where $C_0, C_1 > 0$ are constants independent of
$n\in \N$.
Hence, for any $N', N\in \N$ with $N'>N$, we have 
\begin{align*}
& \Vert w_{N'}- w_N\Vert_{L^1(0,T;X)} 
= \Vert (w_{N'} - w_{N'-1}) + \cdots + (w_{N+1}-w_N)\Vert_{L^1(0,T;X)} \\
\le & \Vert w_{N'} - w_{N'-1}\Vert_{L^1(0,T;X)} 
  + \cdots + \Vert w_{N+1} - w_N\Vert_{L^1(0,T;X)} \\
\le& \sum_{k=N}^{N'-1} \frac{C_0C_1^k}{\Gamma((k+1)\alpha+1)} .
\end{align*}
Stirling's formula yields that 
$$
\lim_{N',N\to\infty} \Vert w_{N'} - w_N\Vert_{L^1(0,T;X)} = 0.
$$
Therefore, there exists $w \in L^1(0,T;X)$ such that
$\lim_{N\to\infty} w_N = w=:u'_\la$ in $L^1(0,T;X)$. This means that $u_{\la} \in W^{1,1}(0,T;X)$. The uniqueness of $u_\la$ follows from the coercivity Lemma \ref{lem5.2}.

Moreover, the Sobolev embedding implies that 
$\Vert u_{\la}\Vert_{L^{\infty}(0,T;X)} < \infty$.
Consequently, 
\begin{align*}
& \left\Vert \int^t_0 \frac{1}{\Gamma(\alpha)}
(t-s)^{\alpha-1}\AAAA u_{\la}(s) ds\right\Vert_X  
\le C\int^t_0 (t-s)^{\alpha-1}\Vert \AAAA\Vert_{X\to X}
\Vert u_{\la}(s)\Vert_X ds\\  
\le & C\Vert \AAAA\Vert_{X\to X} 
\Vert u_{\la}\Vert_{L^{\infty}(0,T;X)} \int^t_0 (t-s)^{\alpha-1}ds
\, \longrightarrow 0
\end{align*}
as $t \to 0$, which implies $u_{\la}(0) = a$.
$\blacksquare$
\smallskip
\\
{\bf Second Step: uniform boundedness in $\la$ of $u_{\la}$.}
\\
Since $u_{\la} \in W^{1,1}(0,T;X)$ and $u_{\la}(0) = a$, 
we see that $\pppa (u_{\la} - a) = \ddda (u_{\la} - a)$
(e.g., \cite{KRY}).
By $\ddda (u_{\la} - a) = \ddda u_{\la} - \ddda a = \ddda u_{\la}$, 
we have
\begin{equation}\label{(5.2)}
\pppa (u_{\la} - a) = \ddda (u_{\la}-a) = \ddda u_{\la}.
\end{equation}
Therefore, \eqref{(5.1)} is equivalent to
\begin{equation}\label{(5.3)}
\ddda u_{\la} = \AAAA u_{\la}, \quad u_{\la}(0) = a.
\end{equation}
Applying Lemmata \ref{lem5.2} and \ref{lem5.1} (v) to \eqref{(5.3)}, we obtain
\begin{align*}
& \frac{1}{2}\Vert u_{\la}(t)\Vert^2_X
- \frac{1}{2}\Vert a\Vert_X^2 \le (J^{\alpha}(\ddda u_{\la},\,
u_{\la})_X)(t) \\
&= \frac{1}{\Gamma(\alpha)}\int^t_0 (t-s)^{\alpha-1}
(\AAAA u_{\la}(s),\, u_{\la}(s))_X ds \le 0, \quad 0<t<T.
\end{align*}
Hence, we deduce
\begin{equation}\label{(5.4)}
\Vert u_{\la}(t)\Vert_X \le \Vert a\Vert_X \quad
\mbox{for $0<t<T$.}
\end{equation}
Thus, we proved:
\begin{lem}\label{lem5.4}
$$
\sup_{\la>0} \Vert u_{\la}(t)\Vert_X 
\le \Vert a\Vert_X< \infty \quad \mbox{for all $0<t<T$.}
$$
\end{lem}

Lemma \ref{lem5.4} implies
$$
\sup_{\la>0} \Vert u_{\la}\Vert_{L^2(0,T;X)} \le T^{\frac{1}{2}}\Vert a\Vert_X
< \infty.
$$
Since $L^2(0,T;X)$ is reflexive, we can find $u \in L^2(0,T;X)$ and extract a subsequence indexed by $\{ \la_n\}_{n\in \N}\subset \{\la>0\}, \, \lambda_n \to \infty,$ such that
\begin{equation}\label{(5.5)}
u_n:= u_{\la_n} \, \rightharpoonup \, u \quad 
\mbox{weakly in $L^2(0,T;X)$}.
\end{equation}
Furthermore, the property of the weak limit yields
\begin{equation}\label{(5.6)}
\Vert u\Vert_{L^2(0,T;X)} \le \liminf_{n\to \infty}
\Vert u_n\Vert_{L^2(0,T;X)} \le T^{\frac{1}{2}}\Vert a\Vert_X.
\end{equation}
\\
{\bf Third Step.}
\\
Finally, we have to verify that $u$ is a weak solution. Write $A_n:= A_{\la_n}$ and $u_n := u_{\la_n}$. Then, since $u_n$ satisfies $\pppa (u_n -a) = A_nu_n$ in $X$ for $0<t<T$, we obtain
\begin{equation}\label{(5.7)}
( \pppa (u_n-a), \, \va)_{\LTLTX}
= (A_nu_n, \, \va)_{\LTLTX}
\end{equation}
for all $\va \in \DDD((\pppa)^*) \cap L^2(0,T;\DDD(A^*))$.
By $\va\in \DDD((\pppa)^*)$, we see that 
$$
\mbox{[the left-hand side of \eqref{(5.7)}]}
= (u_n-a, (\pppa)^*\va)_{\LTLTX}.
$$
Since $A_n$ is a bounded operator on $X$, we see
$$
\mbox{[the right-hand side of \eqref{(5.7)}]}
= (u_n, A_n^*\va)_{\LTLTX}.
$$
Hence,
\begin{equation}\label{(5.8)}
(u_n-a,\, (\pppa)^*\va)_{\LTLTX} = (u_n,\, A_n^*\va)_{\LTLTX}
\end{equation}
for $n\in \N$. In view of \eqref{(5.5)}, letting $n\to \infty$, we obtain
$$
(u-a,\, (\pppa)^*\va)_{\LTLTX} = \lim_{n\to \infty}(u_n,\, A_n^*\va)_{\LTLTX}
$$
for all $\va \in \DDD((\pppa)^*) \cap L^2(0,T;\DDD(A^*))$.

On the other hand, we can consider
$$
 \mbox{[the right-hand side of \eqref{(5.8)}]}
 = (u_n, A^*\va)_{\LTLTX} + (u_n, \, A_n^*\va - A^*\va)_{\LTLT},
$$
and
$$
\lim_{n\to \infty}  \mbox{[the right-hand side of \eqref{(5.8)}]}
= (u, A^*\va)_{\LTLTX} 
+ \lim_{n\to\infty} (u_n, \, A_n^*\va - A^*\va)_{\LTLT}.
$$
Moreover,
$$
\lim_{n\to \infty}\Vert A_n^*\va - A^*\va\Vert_{\LTLTX} = 0.
$$
Indeed, we have
$$
\Vert A_n^*\va - A^*\va\Vert_{\LTLTX}^2
= \int^T_0 \Vert A_n^*\va(t) - A^*\va(t)\Vert^2_X
dt.
$$
We see that $\va(t) \in \DDD(A^*)$ for almost all 
$t\in (0,T)$ by $\va\in L^2(0,T;\DDD(A^*))$.
Hence, Lemma \ref{lem5.1} (iii) implies that 
$A^*_n \va(t) - A^*\va(t) \longrightarrow 0$ 
as $n\to \infty$ for almost all $t\in (0,T)$.
Furthermore, 
\begin{align*}
\Vert A_n^*\va(t) - A^*\va(t)\Vert_X
&\le \Vert A_n^*\va(t)\Vert_X + \Vert A^*\va(t)\Vert_X;\\
\Vert A_n^*\va(t)\Vert_X 
&= \Vert \la_n(\la_n-A^*)^{-1}A^*\va(t)\Vert_X  \\
&\le \la_n\Vert (\la_n - A^*)^{-1}\Vert_{X\to X} \Vert A^*\va(t)
\Vert_X
\le \Vert A^*\va(t)\Vert_X
\end{align*}
by \eqref{(1.2)} for $A^*$, since $(\la_n - A^*)^{-1}=((\la_n - A)^{-1})^*$.
Therefore, 
$$
\Vert A_n^*\va(t) - A^*\va(t)\Vert_X
\le 2\Vert A^*\va(t)\Vert_X \in L^2(0,T)
$$
by $A^*\va \in \LTLTX$. Consequently, the Lebesgue dominated convergence theorem implies 
$$
\lim_{n\to \infty} \Vert A_n^*\va - A^*\va\Vert_{\LTLTX}
= 0,
$$
so that the right-hand side of \eqref{(5.7)} converges to $(u, A^*\va)_{\LTLTX}$ as $n \to \infty$ by using
\begin{lem}\label{lem5.5}
Let $u_n \rightharpoonup u$ weakly in $\LTLTX$ and $\lim_{n\to\infty} \Vert w_n - w\Vert_{\LTLTX} = 0$.
Then,
\\
$\lim_{n\to \infty} (u_n,\, w_n)_{\LTLTX}
= (u,\, w)_{\LTLTX}$.
\end{lem}

Consequently,
$$
(u-a,\, (\pppa)^*\va)_{\LTLTX}
= (u, \, \MAAA^*\va)_{\LTLTX}
$$
for all $\va \in \DDD((\pppa)^*) \cap \DDD(\MAAA^*)$.

Now we prove \eqref{(2.4)}. Let $T>0$ be arbitrary, and define
$$
\mathcal{U}:=\{v\in L^2(0,T;X); \|v(t)\|_X \le \|a\|_X \text{ for a.e. } 
t\in (0,T) \}.
$$
The set $\mathcal{U}$ is convex and closed in $L^2(0,T;X).$ Indeed, if $u_n \to u$ in $L^2(0,T;X),$ there exists a subsequence $\{u_{n_k}\}$ such that $u_{n_k}(t) \to u(t)$ in $X$ for almost every $t\in (0,T)$ (see, e.g., \cite[Lemma 3.22(1), p.~57]{Alt16}). Hence, $\mathcal{U}$ is weakly closed. On the other hand, let $u_n$ be as in \eqref{(5.5)} such that $u_n \, \rightharpoonup \, u$ weakly in $L^2(0,T;X).$ In view of \eqref{(5.4)}, we have $u_n \in \mathcal{U}$ for every $n\in \N.$ Therefore, $u\in \mathcal{U}$.
Moreover, $u\in C([0,T];X)$ follows from Theorem \ref{thmcont} whose proof is independently done of Theorem \ref{thm1}.
$\blacksquare$

Thus, we have proved that Condition $(\mathcal{M})$ yields the 
existence of a weak solution.
\smallskip

\noindent{\bf 4.2. Proof of the necessity of Condition $(\mathcal{M})$.}
\\
{\bf First Step.}
\\
We will prove that $\la - A$ is injective for all $\la > 0$. Let $\va \in \DDD(A)$ satisfy $A\va = \la \va$ for $\la>0$.
Set $u(t) := E_{\alpha,1}(\la t^{\alpha})\va$ for $t>0$. Then, we can easily verify that $\ddda u = Au$ and 
$u(0) = \va$, that is, $\pppa (u-\va) = Au$. We assume that $\va \ne 0$, that is,  $\Vert \va\Vert \ne 0$.

The assumption \eqref{(2.4)} implies $\Vert E_{\alpha,1}(\la t^{\alpha})\va\Vert \le \Vert \va\Vert$ for all $t >0$. Then, $\Vert \va\Vert
\ne 0$ yields
$$
\vert E_{\alpha,1}(\la t^{\alpha})\vert \le 1 \quad 
\mbox{for all $t>0$.}
$$
This is impossible because $E_{\alpha,1}(\la t^{\alpha})>1$ for $t>0$ and $\la >0$. Therefore, $\va = 0$, 
that is, $\la - A$ is injective for $\la > 0$.
\\
{\bf Second Step.}
\\
We set $\whwh{u}(p) := \int^{\infty}_0 e^{-pt}u(t) dt$ for the Laplace transform, provided that the 
right-hand side is defined.

Let $a \in \DDD(A).$ Theorem \ref{cor1.4}~(i) implies that the weak solution to \eqref{(2.2)} is a strong solution. By \eqref{(2.4)}, we can apply the Laplace transform to $\pppa (u_a(t)-a)$
(e.g., Yamamoto \cite[Theorem 2.7.2]{Ya26}). Since $\pppa (u_a - a) = Au_a$ in $(0,T)$ for 
any $T > 0$ and $A$ is closed, using \eqref{(2.4)}, we obtain
$$
\whwh{u}_a(p) \in \DDD(A),\qquad p^{\alpha}\whwh{u}_a(p) - p^{\alpha-1}a = A\whwh{u}_a(p)
\quad \mbox{for all $p>0$}, 
$$
that is, setting $\la = p^{\alpha}$, we have 
$$
(\la - A)\whwh{u}_a(\la^{\frac{1}{\alpha}}) 
= \la^{\frac{\alpha-1}{\alpha}}a
\quad \mbox{for all $\la > 0$.}
$$
Therefore, $\DDD(A) \subset \mathcal{R}(\la -A),$ the range of $\la -A$. Since $\la - A$ is injective, for all $a \in \DDD(A)$ and $\la>0$,
we infer 
$$
(\la - A)^{-1}a = \la^{\frac{1-\alpha}{\alpha}}\whwh{u}_a
(\la^{\frac{1}{\alpha}}).
$$
Since 
$$
\Vert \whwh{u}_a(\la^{\frac{1}{\alpha}})\Vert 
= \left\Vert \int^{\infty}_0 \exp(-\la^{\frac{1}{\alpha}}t)
u_a(t) dt \right\Vert 
\le \Vert a\Vert \int^{\infty}_0 \exp(-\la^{\frac{1}{\alpha}}t) dt 
= \Vert a\Vert \la^{-\frac{1}{\alpha}},
$$
by \eqref{(2.4)}, we obtain
\begin{equation}\label{(2)}
\Vert (\la - A)^{-1}a\Vert \le \frac{1}{\la}\Vert a\Vert
\quad \mbox{for all $a\in \DDD(A)$.}
\end{equation}

Next, we prove that $\mathcal{R}(\la -A)=X.$ Let $x\in X$ and choose $a_n\in \DDD(A)$ such that $a_n \to x$ ($\DDD(A)$ is dense in $X$). By \eqref{(2)}, we see that $(\la - A)^{-1}a_n$ is a Cauchy sequence in $X$. Then $(\la - A)^{-1}a_n \to y \in X.$ On the other hand, $(\la - A)^{-1}a_n\in \DDD(A)$ and $(\la - A) [(\la - A)^{-1}a_n] \to x.$ Since $\la - A$ is closed, $y\in \DDD(A)$ and $(\lambda -A)y=x.$

Therefore, $\{ \la \in \R; \, \la> 0\} \subset \rho(A).$ Since $\DDD(A)$ is dense in $X$, \eqref{(2)} holds for all $a \in X$.
Thus we obtain Condition $(\mathcal{M})$.
$\blacksquare$
\smallskip

\section{Proof of Theorem \ref{thmcont}}\label{sec5}

For $\lambda>0$, let $u_\lambda^a$ be the solution of the approximating problem
\begin{equation}\label{eq:approx-cont}
    \partial_t^\alpha(u_\lambda^a-a)
        =A_\lambda u_\lambda^a,
    \qquad
    u_\lambda^a(0)=a.
\end{equation}
We recall that $u_\lambda^a\in W^{1,1}(0,T;X)\subset C([0,T];X)$ and
\begin{equation}
    \|u_\lambda^a(t)\|_X\leq \|a\|_X,
    \qquad 0\leq t\leq T.
    \label{eq:approx-contraction}
\end{equation}

We first consider initial data $b\in \DDD(A^2)$. For such $b$, one has
\begin{equation}
    \|A^2u_\lambda^b(t)\|_X
        \leq \|A^2b\|_X,
    \qquad 0\leq t\leq T.
    \label{eq:A2-est}
\end{equation}
Indeed, we can show that $A_\lambda \DDD(A^2) \subset \DDD(A^2)$ and $ A^2 A_\lambda =A_\lambda A^2$
on $\DDD(A^2)$. Hence, $u_\lambda^b(t) \in \DDD(A^2)$ and $A^2u_\lambda^b$ satisfies
\[
    \partial_t^\alpha
       \left(A^2u_\lambda^b-A^2b\right)
       =A_\lambda A^2u_\lambda^b,
\]
and \eqref{eq:A2-est} follows from the contraction estimate \eqref{eq:approx-contraction}.

Moreover, for every $c\in \DDD(A^2)$,
\begin{equation}
    \|A_\lambda c-Ac\|_X
        \leq \frac{1}{\lambda}\|A^2c\|_X.
    \label{eq:Yosida-rate}
\end{equation}
Indeed,
$$
    A_\lambda c-Ac=(\mathcal{J}_\lambda-1)Ac,
$$
while, for $c_0\in \DDD(A)$,
$$
    (\mathcal{J}_\lambda-1)c_0
       =(\lambda-A)^{-1}Ac_0.
$$
Since $\|(\lambda-A)^{-1}\|_{X\to X}\leq \frac{1}{\lambda},$ estimate \eqref{eq:Yosida-rate} follows.
\\

Let now $\lambda,\mu>0$.  We set
\[
    v_{\lambda,\mu}
       :=u_\lambda^b-u_\mu^b.
\]
Then
\[
    \partial_t^\alpha v_{\lambda,\mu}
       =(A_\lambda-A_\mu)u_\lambda^b
          +A_\mu v_{\lambda,\mu},
    \qquad
    v_{\lambda,\mu}(0)=0.
\]
Since $A_\mu$ is dissipative by Lemma \ref{lem5.1} $(v)$, 
Lemma \ref{lem5.2} and Young's inequality yield
\begin{align}
    \|v_{\lambda,\mu}(t)\|_X^2
    &\leq
    C\int_0^t
       (t-s)^{\alpha-1}
       \|(A_\lambda-A_\mu)u_\lambda^b(s)\|_X^2\,ds
       \nonumber\\
    &\quad
    +C\int_0^t
       (t-s)^{\alpha-1}
       \|v_{\lambda,\mu}(s)\|_X^2\,ds .
    \label{eq:diff-est}
\end{align}
Here we also used 
$$
\vert ((A_{\lambda}-A_{\mu})u_{\la}^b(s), \, v_{\la,\mu}(s))_X\vert 
\le \frac{1}{2}( \Vert (A_{\lambda}-A_{\mu})u_{\la}^b(s)\Vert_X^2
+ \Vert v_{\la,\mu}(s)\Vert_X^2).
$$
By \eqref{eq:Yosida-rate} and \eqref{eq:A2-est}, we have
\begin{align*}
    \|(A_\lambda-A_\mu)u_\lambda^b(s)\|_X
    &\leq
      \|(A_\lambda-A)u_\lambda^b(s)\|_X
      +\|(A-A_\mu)u_\lambda^b(s)\|_X\\
    &\leq
      \left(\frac1\lambda+\frac1\mu\right)
      \|A^2u_\lambda^b(s)\|_X
    \leq
      \left(\frac1\lambda+\frac1\mu\right)
      \|A^2b\|_X.
\end{align*}
Therefore, the generalized Gronwall inequality applied to
\eqref{eq:diff-est} gives
\begin{equation}\label{eq:Cauchy-regular}
    \|u_\lambda^b-u_\mu^b\|_{C([0,T];X)}
    \leq
    C_{\alpha,T}
    \left(\frac1\lambda+\frac1\mu\right)
    \|A^2b\|_X.
\end{equation}
Thus $\{u_\lambda^b\}_{\lambda>0}$ is Cauchy in $C([0,T];X)$ whenever $b\in \DDD(A^2)$.

We now remove the assumption $b\in \DDD(A^2)$. First, we note that $\DDD(A^2)$ is dense in $X$ (Lemma \ref{lem5.1}~$(vi)$). 
By the linearity and \eqref{eq:approx-contraction}, for any $a,b\in X$,
\begin{equation}
    \|u_\lambda^a-u_\lambda^b\|_{C([0,T];X)}
       \leq \|a-b\|_X.
    \label{eq:data-cont}
\end{equation}
We claim that $\{u_\lambda^a\}_{\lambda>0}$ is Cauchy in $C([0,T];X)$ for every $a\in X$.

Indeed, let $\varepsilon>0$. By the density of $\DDD(A^2)$ in $X$, choose $b\in \DDD(A^2)$ such that $\|a-b\|_X<\varepsilon.$ Then, using \eqref{eq:data-cont},
we have
\begin{align*}
    \|u_\lambda^a-u_\mu^a\|_{C([0,T];X)}
    &\leq
      \|u_\lambda^a-u_\lambda^b\|_{C([0,T];X)}
      +\|u_\lambda^b-u_\mu^b\|_{C([0,T];X)}
      +\|u_\mu^b-u_\mu^a\|_{C([0,T];X)}\\
    &\leq
      2\|a-b\|_X
      +\|u_\lambda^b-u_\mu^b\|_{C([0,T];X)}.
\end{align*}
By \eqref{eq:Cauchy-regular}, we obtain
\[
    \limsup_{\lambda,\mu\to\infty}
    \|u_\lambda^a-u_\mu^a\|_{C([0,T];X)}
    \leq 2\|a-b\|_X
    <2\varepsilon.
\]
Since $\varepsilon>0$ is arbitrary, $\{u_\lambda^a\}_{\lambda>0}$ is Cauchy in $C([0,T];X)$.
Consequently, there exists $\widetilde u\in C([0,T];X)$ such that
\begin{equation}
    u_\lambda^a\longrightarrow\widetilde u
    \qquad\text{in }C([0,T];X).
    \label{eq:uniform-conv}
\end{equation}
In particular, $u_\lambda^a\rightarrow\widetilde u \text{ strongly in }L^2(0,T;X).$

We now identify $\widetilde u$ with the weak solution $u$.
For every $\va\in D((\partial_t^\alpha)^*)\cap L^2(0,T;D(A^*)),$ the approximating solutions satisfy
\[
    (u_\lambda^a-a,(\partial_t^\alpha)^*\va)_{L^2(0,T;X)}
       =
    (u_\lambda^a,A_\lambda^*\va)_{L^2(0,T;X)}.
\]
Since $u_\lambda^a\to\widetilde u
        \quad\text{in }L^2(0,T;X)$ and $A_\lambda^*\va\to A^*\va
        \quad\text{in }L^2(0,T;X),$ letting $\lambda\to\infty$ gives
\[
    (\widetilde u-a,(\partial_t^\alpha)^*\va)_{L^2(0,T;X)}
       =
    (\widetilde u,A^*\va)_{L^2(0,T;X)}.
\]
Thus $\widetilde u$ is a weak solution of \eqref{(2.2)}. By uniqueness of the weak solution, $\widetilde u=u \text{ in }L^2(0,T;X).$
Hence $u\in C([0,T];X)$.

Finally, since $u_\lambda^a(0)=a$ and the convergence \eqref{eq:uniform-conv} is uniform, $u(0)=a.$ This completes the proof. 
$\blacksquare$
\smallskip

\section{Proof of Theorem \ref{cor1.4}}\label{sec6}

(i) Let $a\in \DDD(A)$. Then, by the assumption of the theorem, 
it follows that there exists a unique weak solution $v\in L^2(0,T;X)$ corresponding to the initial datum $Aa$.
More precisely,
$$
(v-Aa,\, (\pppa)^*\psi)_{\LTLTX} 
= (v,\, \mathcal{A}^*\psi)_{\LTLTX} \quad \mbox{for all $\psi 
\in \DDD((\pppa)^*) \cap \DDD(\mathcal{A}^*)$}.
$$

Define $z:=a+J^\alpha v.$ Then $z-a=J^\alpha v\in H_\alpha(0,T;X)$
and
\begin{equation}\label{eqlm1}
\partial_t^\alpha(z-a)=v.
\end{equation}
We next prove that $z\in \DDD(\mathcal{A})$ and $\mathcal{A}z=v$. Let $\va\in \DDD(\mathcal{A}^*)=L^2(0,T;\DDD(A^*))$ be arbitrary and set $\psi:=J_\alpha\va$. Since $\DDD((\partial_t^\alpha)^*)=J_\alpha L^2(0,T;X),$
we have
\[
    \psi\in \DDD((\partial_t^\alpha)^*) \qquad\mbox{and}\quad 
    (\partial_t^\alpha)^*\psi=\va.
\]
Moreover, since $J_\alpha$ acts only on the time variable, it commutes with $\mathcal{A}^*$, and hence
\[
    \psi\in \DDD(\mathcal{A}^*) \qquad\mbox{and}\quad
    \mathcal{A}^*\psi=J_\alpha \mathcal{A}^*\va.
\]
Using $\psi$ as a test function in the weak formulation satisfied by
$v$, we obtain
\begin{equation}\label{eqlm2}
    (v-Aa,\va)_{L^2(0,T;X)}
    =
    (v,J_\alpha \mathcal{A}^*\va)_{L^2(0,T;X)}.
\end{equation}
On the other hand, since $(J^\alpha)^*=J_\alpha,$ we have
\[
    (J^\alpha v,\mathcal{A}^*\va)_{L^2(0,T;X)}
    =
    (v,J_\alpha \mathcal{A}^*\va)_{L^2(0,T;X)}.
\]
Combining this identity with \eqref{eqlm2}, we obtain
\[
    (J^\alpha v,\mathcal{A}^*\va)_{L^2(0,T;X)}
    =
    (v-Aa,\va)_{L^2(0,T;X)}.
\]
Therefore, since $z=a + J^{\alpha}v$, we have
\begin{align*}
    (z,\mathcal{A}^*\va)_{L^2(0,T;X)}
    &=
    (a,\mathcal{A}^*\va)_{L^2(0,T;X)}
    +(J^\alpha v,\mathcal{A}^*\va)_{L^2(0,T;X)} \\
    &=
    (Aa,\va)_{L^2(0,T;X)}
    +(v-Aa,\va)_{L^2(0,T;X)} \\
    &=
    (v,\va)_{L^2(0,T;X)}.
\end{align*}
Since this holds for every $\va\in \DDD(\mathcal{A}^*)$, the definition of
$(\mathcal{A}^*)^*$ yields
\[
    z\in \DDD((\mathcal{A}^*)^*),
    \qquad
    (\mathcal{A}^*)^*z=v.
\]
Since $\mathcal{A}$ is densely defined and closed, it follows that 
$(\mathcal{A}^*)^*=\mathcal{A}.$ Consequently,
\begin{equation}\label{eqlm3}
    z\in \DDD(\mathcal{A}),
    \qquad
    \mathcal{A}z=v.
\end{equation}
Combining \eqref{eqlm1} and \eqref{eqlm3}, we obtain
\[
    \partial_t^\alpha(z-a)=Az.
\]
Therefore, $z$ is a strong solution to \eqref{(2.2)} with initial datum $a$. In particular, $z$ is also a weak solution. By the uniqueness of the weak solution corresponding to $a$, we have $z=u_a.$ This completes the proof.

(ii) It can be easily proved by induction, repeating the arguments in 
Theorem \ref{cor1.4} (i).
$\blacksquare$
\smallskip

\section{Proof of Theorem \ref{thmpert}}\label{sec7}

Consider the Yosida approximation of the problem, i.e.,
\begin{equation}\label{ypert}
\pppa (\uull -a ) = \AAAA\uull + B \uull.
\end{equation}
Similarly to Theorem \ref{thm1}, one can prove that the solution $\uull \in W^{1,1}(0,T; X)$ to \eqref{ypert} uniquely exists.

For $\uull \in W^{1,1}(0,T;X)$, we see that \eqref{ypert} is equivalent to 
$$
\ddda \uull = \AAAA\uull + B\uull, \qquad \uull(0) = a.
$$
Hence, 
$$
(\ddda \uull(s), \, \uull(s))_X 
= (\AAAA\uull(s), \, \uull(s))_X + (B\uull(s),\,\uull(s))_X.
$$
We note that:
$$
\mbox{if $f(s) \ge g(s)$ for $0\le s \le T$, then 
$(J^{\alpha}f)(t) \ge (J^{\alpha}g)(t)$ for $0\le t \le T$.}
$$
Since $(\AAAA\uull(s), \, \uull(s))_X \le 0$, we have
\begin{equation}\label{(eq3)}
(J^{\alpha}(\ddda \uull(s), \, \uull(s))_X)(t)\le (J^{\alpha}(B\uull,\,\uull)_X)(t).
\end{equation}
On the other hand, since 
$$
\vert (B\uull, \, \uull)_X\vert \le \Vert B\uull\Vert_X\Vert \uull\Vert_X
\le C\Vert \uull\Vert^2_X,
$$
we obtain
$$
(J^{\alpha}(B\uull,\, \uull)_X)(t)
\le CJ^{\alpha}(\Vert \uull(s)\Vert_X^2)(t)
= C\int^t_0 (t-s)^{\alpha-1}\Vert \uull(s)\Vert^2_X ds
$$
In \eqref{(eq3)}, applying this and the coercivity (Lemma \ref{lem5.2}), we have
$$
\Vert \uull(t)\Vert^2_X \le \Vert a\Vert^2_X + C\int^t_0 (t-s)^{\alpha-1}\Vert \uull(s)\Vert^2_X ds.
$$
The generalized Gronwall inequality implies 
$$
\Vert \uull(t)\Vert^2_X \le C\Vert a\Vert^2_X \quad \mbox{for almost all 
$t\in (0,T)$}.
$$
As before, passing to the limit, we obtain that the weak limit of a subsequence of
$\{ u_{\la}\}$ is a weak solution. This yields the corresponding estimate of the solution. The uniqueness of the solution can be proved similarly to Section \ref{sec3} using the generalized Gronwall inequality. $\blacksquare$

\section{Concluding remarks}\label{sec8}

In this article, we discuss initial value problems for time-fractional evolution equations $\pppa (u-a) = Au + Bu$, where $A$ is a generator of a contraction C$_0$ semigroup in a real Hilbert space $X$ and $B$ is a bounded linear operator on $X$, and establish the well-posedness. Our main result is a generalization of the Hille-Yosida theorem. The key of the proof is the Yosida approximation, but unlike $\alpha=1$, we cannot rely on the structure of the exponential function, so we prove uniform boundedness of approximating solutions, which is similar to the Galerkin method.

The key ideas can be described as follows:
\begin{itemize}
\item 
construct approximating solutions using the Yosida approximation;
\item 
prove uniform boundedness of approximating solutions in $L^2(0,T;X)$, which yields a weakly convergent subsequence. The weak limit is then proved to be a weak solution.
\end{itemize}

The first step of approximating the original equation by a finite-dimensional or a well-posed system admits several possibilities:
\begin{itemize}
\item
Galerkin method;
\item
elliptic regularization;
\item
Yosida approximation.
\end{itemize}

The Galerkin method is effective especially in the case where $A$ is of elliptic type or $A$ is a generator of 
an analytic C$_0$ semigroup, and for the fractional partial differential equations, we can refer, for example, to Kubica, Ryszewska and Yamamoto \cite{KRY}. However, in the case where $A$ is in a class of generators of C$_0$ semigroups, an application of 
the Galerkin method must overcome several difficulties. Although we can apply the elliptic regularization,
our proof relies on the Yosida approximation (e.g., \cite{Pa}, \cite{Ta}).

\section{Proofs of useful Lemmata}\label{sec9}
First, we prove
\begin{lem}\label{lem9.1}
Let $A$ satisfy Condition $(\mathcal{M})$. Then,
$$
(Aw,w)_X \le 0 \quad \mbox{for all $w\in \DDD(A)$.}
$$
\end{lem}
The lemma is known, but for convenience, we provide a proof.

\noindent{\bf Proof of Lemma \ref{lem9.1}}
\\
Let $\la > 0$ be arbitrary.  We have
$$
\la\Vert (\la-A)^{-1}v\Vert_X \le \la \Vert (\la-A)^{-1}\Vert_{X\to X}
\Vert v\Vert_X \le \Vert v\Vert_X
$$
for all $v \in X.$ Setting $w:= (\la-A)^{-1}v, \, v \in X$, we see that 
$w\in \DDD(A)$ and $\la\Vert w\Vert_X \le \Vert (\la-A)w\Vert_X$.
Then, $\Vert (A-\la)w\Vert^2_X \ge \la^2\Vert w\Vert^2_X$, and so
$$
0 \le \Vert(A-\la)w\Vert^2_X - \la^2\Vert w\Vert^2_X
= \Vert Aw\Vert^2_X - 2\la(Aw,w)_X.
$$
Hence, $(Aw,w)_X \le \frac{1}{2\la}\Vert Aw\Vert^2_X$ for 
$w\in \DDD(A)$ and any $\la > 0$.
Letting $\la \to \infty$, we see that $(Aw,w)_X \le 0$.
This proves Lemma \ref{lem9.1}.
$\blacksquare$
\smallskip

\noindent{\bf Proof of Lemma \ref{lem3.1} (i).}
\\
As in \cite[Theorem 2.5]{KRY}, we see that $\partial_t^\alpha$ is closed. The closedness of $\mathcal{A}$ follows easily from the closedness of $A$ and the fact that the convergence in $L^2(0,T;X)$ implies the pointwise convergence of a subsequence in $X$ (e.g., Lemma 3.22 (1) in \cite{Alt16}). $\blacksquare$
\smallskip

Lemma \ref{lem3.1} (ii)-(iii) can be proved as in \cite{FGY25}.
\smallskip

\noindent{\bf Proof of Lemma \ref{lem3.1} (iv).}
\\
Let $\va \in \LTLTX$ be given arbitrarily.
As is shown easily, $\DDD(\MAAA^*) = L^2(0,T;\DDD(A^*))$ is 
dense in $\LTLTX$ by using $\ooo{\DDD(A^*)^X} = X$.
Therefore, for any $\ep > 0$, we can find $\va_{\ep}\in \DDD(\MAAA^*)$
such that $\Vert \va - \va_{\ep}\Vert_{\LTLTX} < \ep$.
By Lemma \ref{lem3.1} (v), for $\va_{\ep}$, there exists $\psi_{\ep}
\in \DDD(\MAAA^*)\cap \DDD((\pppa)^*)$ such that 
$$
\Vert \va_{\ep} - \psi_{\ep}\Vert_{L^2(0,T;\DDD(A^*))} < \ep.
$$
On the other hand, 
$$
\Vert \va_{\ep} - \psi_{\ep}\Vert_{\LTLTX} 
\le C\Vert \va_{\ep} - \psi_{\ep}\Vert_{L^2(0,T;\DDD(A^*))},
$$
where the constant $C>0$ does not depend on $\va_{\ep}, \psi_{\ep}$.
Hence,
$$
\Vert \va - \psi_{\ep}\Vert_{\LTLTX}
\le \Vert \va - \va_{\ep}\Vert_{\LTLTX}
+ \Vert \va_{\ep} - \psi_{\ep}\Vert_{\LTLTX}
\le (1+C)\ep
$$
with $\psi_{\ep} \in \DDD(\MAAA^*) \cap 
\DDD((\pppa)^*)$.
Thus the proof of Lemma \ref{lem3.1} (iv) is complete.
$\blacksquare$
\smallskip

\noindent{\bf Proof of Lemma \ref{lem3.1} (v).}
\\
Let $\va \in \DDD(\MAAA^*)$ be arbitrarily given.
For arbitrary $\delta \in (0,T/2)$, we define $\va_{\delta} \in L^2(\R; \DDD(A^*))$
such that  
$$
\va_{\delta}(t) = 
\left\{ \begin{array}{rl}
& 0, \	\quad T-\delta < t < \infty \,\, \mbox{or}\,\,
-\infty < t < \delta,\\
& \va(t), \quad \mbox{otherwise}.
\end{array}\right.
$$
Then, 
$$
\Vert \va - \va_{\delta}\Vert_{L^2(0,T;\DDD(A^*))}
= \left(\left(
\int^{\delta}_0 + \int^T_{T-\delta} \right)
\Vert \va(t)\Vert_{\DDD(A^*)}^2 dt \right)^{\frac{1}{2}}.
$$
Since $\Vert \va(\cdot)\Vert_{\DDD(A^*)} \in L^2(0,T)$, 
for any $\ep > 0$ there exists a constant $\delta = \delta(\ep) > 0$ 
such that  
$$
\Vert \va - \va_{\delta}\Vert_{L^2(0,T;\DDD(A^*))} < \ep.
$$
Now, we introduce the mollifier (e.g., Adams \cite{Ad}).
More precisely, we choose $\chi \in C^{\infty}_c(\R)$ satisfying
$\chi(t) \ge 0$ for $t \in \R$, supp $\chi \subset \{ -1< t < 1\}$ and
$\int^{\infty}_{-\infty} \chi(t) dt = 1$.
For constant $\mu > 0$, we set 
$$
\www{\va_{\delta,\mu}}(t) := \int^{\infty}_{-\infty}
\frac{1}{\mu}\chi\left( \frac{t-s}{\mu} \right)\va_{\delta}(s) ds,
\quad t\in \R.
$$
Then, we can find $\mu = \mu(\ep) < \delta(\ep)$ such that 
$\www{\va_{\delta,\mu}} \in C^{\infty}_c((0,T);\DDD(A^*))
\subset \DDD((\pppa)^*) \cap \DDD(\MAAA^*)$,
supp $\www{\va_{\delta,\mu}} \subset (0,T)$,
and 
$$
\Vert \va_{\delta} - \www{\va_{\delta,\mu}}\Vert_{L^2(0,T;\DDD(A^*))}
< \ep.
$$
Setting $\www{\va_{\ep}}:= \www{\va_{\delta,\mu}} = \www{\va_{\delta(\ep),
\mu(\ep)}}$, we see that 
$\www{\va_{\ep}} \in \DDD((\pppa)^*) \cap \DDD(\MAAA^*)$ and
\begin{align*}
& \Vert \www{\va_{\ep}} - \va\Vert_{L^2(0,T;\DDD(A^*))}
= \Vert \www{\va_{\ep}} - \va_{\delta} + \va_{\delta} - \va
\Vert_{L^2(0,T;\DDD(A^*))}\\
\le & \Vert \www{\va_{\ep}} - \va_{\delta}\Vert_{L^2(0,T;\DDD(A^*))}
 + \Vert \va_{\delta} - \va\Vert_{L^2(0,T;\DDD(A^*))}
\le \ep + \ep.
\end{align*}
Thus the proof of Lemma \ref{lem3.1} (v) is complete.
$\blacksquare$
\smallskip

\noindent{\bf Proof of Lemma \ref{lem3.2}.}
\\
First, we prove
\begin{equation}\label{(7.4)}
(J^{\alpha})^* = J_{\alpha}.
\end{equation}
Indeed, for $u,v \in \LTLTX$, exchanging the orders of the integrals,
we obtain
\begin{align*}
& (J^{\alpha}u,v)_{\LTLTX}
= \frac{1}{\Gamma(\alpha)}\int^T_0 \left( \int^t_0 (t-s)^{\alpha-1}
u(s) ds, v(t)\right)_X dt\\
&= \frac{1}{\Gamma(\alpha)}\int^T_0 \left(u(s), \int^T_s 
(t-s)^{\alpha-1} v(t) dt \right)_X ds
= \int^T_0 (u(s), J_{\alpha}v(s))_X ds = (u,\, J_{\alpha}v)
_{\LTLTX}.
\end{align*}
This proves \eqref{(7.4)}.

Moreover, by the same way as in the proof of Lemma \ref{lem3.1} (v), based on 
the mollifier, arguing in 
$X$, not in $\DDD(\MAAA^*)$, for arbitrary $\va \in L^2(0,T;X)$, 
we can find $\www{\va_{\ep}} \in 
C^{\infty}_c(0,T;X)$ such that $\Vert \va - \www{\va_{\ep}}\Vert_{\LTLTX}
< \ep$.
Then, $\www{\va_{\ep}} \in J_{\alpha}\LTLTX$.
\\
Indeed, using $J_1 = J_{\alpha}J_{1-\alpha}$, we can obtain
\begin{align*}
& \www{\va_{\ep}}(t) = -\int^T_t \www{\va_{\ep}}'(s) ds
= (J_1(-\www{\va_{\ep}}'))(t) \\
&= J_{\alpha}(J_{1-\alpha}(-\www{\va_{\ep}}'))(t)
\in J_{\alpha}\LTLTX,
\end{align*}
which means that there exists $\www{\va_{\ep}} \in J_{\alpha}\LTLTX$ 
satisfying $\Vert \va - \www{\va_{\ep}}\Vert_{\LTLTX} < \ep.$ Thus, we proved Lemma \ref{lem3.2} (i). 

Next, with Lemma \ref{lem3.2} (i), 
we can apply a theorem on the inverse of an adjoint operator
(e.g., Theorem 1 on p.~224 in Yosida \cite{Yo}) to $J^{\alpha}$ to 
obtain $((J^{\alpha})^{-1})^* = ((J^{\alpha})^*)^{-1}$. We note that Theorem 1 in \cite{Yo} is considered for a dual operator in 
a Banach space, but also holds for the adjoint in a Hilbert space.
Therefore, \eqref{(7.4)} implies that $(\pppa)^* = (J_{\alpha})^{-1}$ and 
$\DDD((\pppa)^*) = J_{\alpha}\LTLTX$.
Thus, the proof of Lemma \ref{lem3.2} is complete.

Part (iii) is directly verified by exchanging the orders of the integrals
in $(J^{\alpha}f,g)_{\LTLTX}$.
\smallskip
 
\noindent{\bf Proof of Lemma \ref{lem5.1} (v).}
\\
We have
$$
(\AAAA u, u)_X = ((-\la + \la\JJJJ)u,u)_X
= -\la\Vert u\Vert^2_X + \la(\JJJJ u,u)_X.
$$
Since $\Vert \JJJJ\Vert_{X\to X} \le 1$, we have 
$\vert \la(\JJJJ u,u)_X\vert \le \la \Vert u\Vert_X^2$, that is,
$$
-\la\Vert u\Vert^2_X \le \la(\JJJJ u, u)_X \le \la\Vert u\Vert^2_X.
$$
Hence, $-\la\Vert u\Vert^2_X + \la(\JJJJ u, u)_X \le 0$,
which implies that $(\AAAA u, u)_X \le 0$.
$\blacksquare$
\smallskip

\noindent{\bf Acknowledgments.}
Masahiro Yamamoto was supported by Grant-in-Aid for Challenging Research (Pioneering) 21K18142 of 
Japan Society for the Promotion of Science.

\end{document}